\documentclass[11pt]{amsart}
\usepackage[margin=1.05in]{geometry}
\usepackage{amsmath,amssymb,amsthm,mathrsfs,mathtools}
\usepackage{enumitem}
\usepackage{booktabs}
\usepackage{tikz}
\usetikzlibrary{arrows.meta,decorations.pathreplacing,calc}
\usepackage{pgfplots}
\pgfplotsset{compat=1.18}
\usepackage[colorlinks=true,linkcolor=blue!50!black,citecolor=green!40!black,urlcolor=blue!60!black]{hyperref}

\newtheorem{theorem}{Theorem}[section]
\newtheorem{lemma}[theorem]{Lemma}
\newtheorem{proposition}[theorem]{Proposition}

\newtheorem*{mainthm}{Main Theorem}
\theoremstyle{definition}

\theoremstyle{remark}
\newtheorem{remark}[theorem]{Remark}
\numberwithin{equation}{section}
\numberwithin{figure}{section}

\newcommand{\R}{\mathbb{R}}

\newcommand{\Z}{\mathbb{Z}}
\newcommand{\I}{[0,1]}
\newcommand{\ent}{\operatorname{ent}}
\newcommand{\htop}{h}
\newcommand{\Spec}{\operatorname{Spec}}
\newcommand{\Fix}{\operatorname{Fix}}
\newcommand{\conv}{\operatorname{conv}}
\newcommand{\supp}{\operatorname{supp}}
\newcommand{\KK}{\mathscr{K}}
\newcommand{\cP}{\mathscr{P}}
\newcommand{\cT}{\mathscr{T}}
\newcommand{\cM}{\mathscr{M}}

\newcommand{\CHD}{\mathrm{CHD}}
\newcommand{\CHL}{\mathrm{CHL}}
\newcommand{\TE}{\mathrm{TE}}
\begin{document}

\title[Entropy spectra of chaotic interval maps]{The entropy spectra of chaotic maps of an interval}

\author{Vinod Kumar P B}
\address{Department of Mathematics, Muthoot Institute of Technology and Science, Kochi, Kerala, India}
\email{vinodkumarpb@mgits.ac.in}


\begin{abstract}
For a continuous self-map $f$ of a real interval $L$ let $\ent(f)$ be the supremum of the topological entropies of the restrictions of $f$ to compact invariant sets. For a class $\cP$ of maps let $\Spec_\cP(L)=\{\ent(f):f\in\cP\}$ be its entropy spectrum. We determine the spectra of the three standard classes of chaotic maps of the compact interval: Devaney chaotic maps have spectrum $[\frac12\log2,\infty]$, maps of positive entropy have spectrum $(0,\infty]$, and Li--Yorke chaotic maps have spectrum $[0,\infty]$; mixing maps have spectrum $(\frac12\log2,\infty]$, and Devaney chaotic maps with at least two fixed points have spectrum $[\log2,\infty]$. For an arbitrary interval, compact or not, maps with a point of least period $n$ have spectrum $[\log\sigma_n,\infty]$ and Li--Yorke chaotic maps have spectrum $[0,\infty]$. The proofs combine a truncation theorem for maps of non-compact intervals, a horseshoe argument showing that two fixed points force entropy $\log2$ on transitive maps, and explicit constructions: an exact piecewise linear family, a mixing map of infinite entropy and a square-root construction. As a consequence, a Devaney chaotic interval map attains the minimal entropy $\frac12\log2$ only if it has a unique fixed point and is not mixing.
\end{abstract}

\maketitle

%=====================================================================
\section{Introduction}\label{sec:intro}
%=====================================================================

Let $\I$ be the compact unit interval and $C(\I)$ the set of its continuous self-maps. Three notions of chaos are in common use for such maps. A map is \emph{Devaney chaotic} if it is topologically transitive and its periodic points are dense \cite{Dev89,BBCDS92}; for interval maps transitivity alone already implies density of periodic points \cite{Sha64b,AKLS99}. A map is \emph{Li--Yorke chaotic} if it has an uncountable scrambled set \cite{LY75}. Finally, positive topological entropy \cite{AKM65,Bow71} is itself often taken as the definition of chaos. The relations between these notions for interval maps are classical: Devaney chaos implies positive entropy, because transitive interval maps have entropy at least $\frac12\log2$ \cite{Blo82,BC87}; positive entropy implies Li--Yorke chaos \cite{BGKM02}; and Li--Yorke chaos does not imply positive entropy \cite{Smi86}. Thus, writing $\CHD$, $\TE$ and $\CHL$ for the classes of Devaney chaotic maps, maps of positive entropy and Li--Yorke chaotic maps,
\begin{equation}\label{eq:classes}
\CHD(\I)\subsetneq\TE(\I)\subsetneq\CHL(\I).
\end{equation}
The interplay between transitivity, dense periodicity and entropy on intervals and related spaces continues to be an active subject. In the closure of the set of interval maps with dense periodic points, positive entropy already forces entropy at least $\frac12\log2$, and infinite entropy is generic \cite{BCOT24}; on compact spaces containing a free arc, transitive non-minimal maps have dense periodic points and positive entropy, with short proofs given in \cite{KW25}; and for circle maps the relation between small entropy of transitive maps and the structure of their sets of periods is studied in \cite{ABG24}.

In this paper we ask a quantitative question: \emph{which values of the entropy occur in each of these classes?} For a class $\cP$ of maps of an interval $L$ we call
\[
\Spec_\cP(L)=\{\ent(f): f\in\cP\}\subset[0,\infty]
\]
the \emph{entropy spectrum} of $\cP$. Here $\ent(f)$ is the entropy of \cite{CR05}, defined as the supremum of the topological entropies of the restrictions of $f$ to its non-empty compact invariant sets; it coincides with the usual topological entropy when $L$ is compact, and it is the appropriate notion for maps of non-compact intervals, since it is independent of the metric (for a comparison of entropy notions for maps of an interval, including piecewise continuous ones, see \cite{CV24}). The infimum $\inf\Spec_\cP(L)$ is the least entropy that the property $\cP$ forces; the classical results quoted above say that $\inf\Spec_{\CHD}(\I)=\frac12\log2$, $\inf\Spec_{\TE}(\I)=0$ and $\inf\Spec_{\CHL}(\I)=0$. The spectrum itself is a finer invariant: it records not only the least entropy compatible with the property but all of them. Our main result determines these spectra completely.

\begin{mainthm}
Let $\sigma_n$ denote the Štefan constants (Section~\ref{sec:prelim}).
\begin{enumerate}[label=(\roman*)]
\item $\Spec_{\CHD}(\I)=[\tfrac12\log2,\infty]$, and the spectrum of mixing maps of $\I$ is $(\tfrac12\log2,\infty]$.
\item The spectrum of Devaney chaotic maps of $\I$ with at least two fixed points is $[\log2,\infty]$. Consequently a Devaney chaotic map of $\I$ of entropy less than $\log2$ has a unique fixed point, and a Devaney chaotic map of entropy exactly $\frac12\log2$ has a unique fixed point and is not mixing.
\item $\Spec_{\TE}(\I)=(0,\infty]$ and $\Spec_{\CHL}(\I)=[0,\infty]$; hence
\[
\Spec_{\CHD}(\I)\subsetneq\Spec_{\TE}(\I)\subsetneq\Spec_{\CHL}(\I).
\]
\item For every interval $L$, compact or not, and every $n\ge1$, the spectrum of maps of $L$ having a point of least period $n$ is $[\log\sigma_n,\infty]$, the spectrum of maps having points of all periods is $[\log\frac{1+\sqrt5}2,\infty]$, and $\Spec_{\CHL}(L)=[0,\infty]$.
\end{enumerate}
\end{mainthm}

The Main Theorem says that in every case the spectrum is an interval ending at $\infty$, whose left endpoint is the forced entropy and belongs to the spectrum exactly for Devaney chaos, two fixed points, periods and Li--Yorke chaos, but not for mixing or for positive entropy. Figure~\ref{fig:spectralines} summarises the statement.

Three ingredients enter the proof, and each has some independent interest.
\begin{itemize}[leftmargin=*]
\item A \emph{truncation theorem} (Theorem~\ref{thm:truncation}): for a map $f$ of a non-compact interval and a compact subinterval $J$, let $f_J$ be the map of $J$ obtained by composing $f$ with the nearest-point retraction onto $J$. Then $\ent(f)\ge\htop(f_J)$, and $\ent(f)$ is the supremum of all $\htop(f_J)$. Clipping creates new orbits but never creates entropy. This is what allows the Štefan lower bound to pass from compact to arbitrary intervals in part (iv).
\item A \emph{two-fixed-point theorem} (Theorem~\ref{thm:interval}): a transitive map of $\I$ with two fixed points has a $2$-horseshoe, hence entropy at least $\log2$. This is a sharp strengthening of the bound $\frac12\log2$ of Blokh and gives the rigidity statement in part (ii).
\item Explicit \emph{constructions} (Section~\ref{sec:constructions}): the piecewise linear maps $Z_\lambda(x)=\operatorname{dist}(\lambda x,2\Z)$, which we show to be exact for $\lambda\ge2$; a mixing map of $\I$ with infinite entropy fixing both endpoints; a square-root construction that halves the entropy of a transitive map; and gluing lemmas that adjust entropy without destroying periodic orbits or scrambled sets.
\end{itemize}

The paper is organised as follows. Section~\ref{sec:prelim}  recalls the results we use. Section~\ref{sec:truncation} proves the truncation theorem, Section~\ref{sec:twofixed} the two-fixed-point theorem, and Section~\ref{sec:constructions} the constructions. Section~\ref{sec:proof} assembles the proof of the Main Theorem.

\begin{figure}[t]
\centering
\begin{tikzpicture}[x=9cm,y=0.7cm]
\foreach \row/\lab in {0/{$\Spec_{\CHD}(\I)$},1/{mixing maps of $\I$},2/{$\CHD$, two fixed points},3/{$\Spec_{\TE}(\I)$},4/{$\Spec_{\CHL}(\I)$},5/{least period $3$, any $L$}}{
  \node[anchor=east,font=\small] at (-0.03,-\row) {\lab};
  \draw[gray!50] (0,-\row)--(1.05,-\row);}
\draw[very thick,blue!70!black] (0.34657,0)--(1.05,0); \fill[blue!70!black] (0.34657,0) circle (2.2pt);
\draw[very thick,blue!70!black] (0.34657,-1)--(1.05,-1); \filldraw[fill=white,draw=blue!70!black,thick] (0.34657,-1) circle (2.2pt);
\draw[very thick,blue!70!black] (0.69315,-2)--(1.05,-2); \fill[blue!70!black] (0.69315,-2) circle (2.2pt);
\draw[very thick,blue!70!black] (0,-3)--(1.05,-3); \filldraw[fill=white,draw=blue!70!black,thick] (0,-3) circle (2.2pt);
\draw[very thick,blue!70!black] (0,-4)--(1.05,-4); \fill[blue!70!black] (0,-4) circle (2.2pt);
\draw[very thick,blue!70!black] (0.48121,-5)--(1.05,-5); \fill[blue!70!black] (0.48121,-5) circle (2.2pt);
\foreach \r in {0,...,5}{\node[font=\scriptsize,blue!70!black] at (1.1,-\r) {$\infty$};}
\draw (0,-5.7)--(1.05,-5.7);
\foreach \v/\l in {0/0,0.34657/{$\frac12\log2$},0.48121/{$\log\frac{1+\sqrt5}2$},0.69315/{$\log2$},1.0/{$1$}}{\draw (\v,-5.6)--(\v,-5.8) node[below,font=\scriptsize]{\l};}
\end{tikzpicture}
\caption{The spectra of the Main Theorem. A filled dot means that the left endpoint belongs to the spectrum, an open dot that it does not.}
\label{fig:spectralines}
\end{figure}
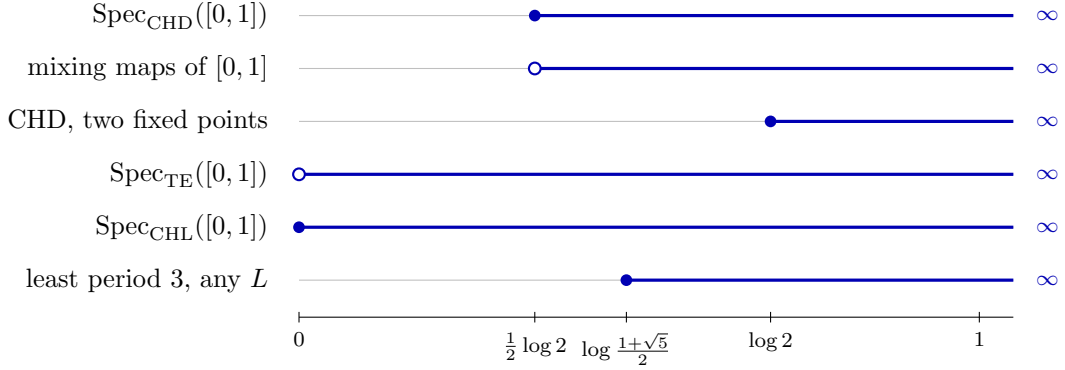

%=====================================================================
\section{Preliminaries}\label{sec:prelim}
%=====================================================================

An \emph{interval} is a connected subset of $\R$ with non-empty interior. For a compact interval $J=[a,b]\subset L$ the \emph{clamp} $r_J\colon L\to J$ is $r_J(x)=\min\{\max\{x,a\},b\}$, and the \emph{truncation} of $f\in C(L)$ is $f_J=r_J\circ f|_J\in C(J)$.

For a compact metric space $Y$ and $g\in C(Y)$, $\htop(g)$ denotes the topological entropy \cite{AKM65,Bow71}. We use the following standard facts \cite{Wal82,ALM00}.

\begin{theorem}\label{thm:compactprops}
Let $Y,Y'$ be compact metric spaces, $g\in C(Y)$, $g'\in C(Y')$.
\begin{enumerate}[label=(\alph*)]
\item If $\varphi\colon Y\to Y'$ is continuous and $g'\varphi=\varphi g$, then $\htop(g)\le\htop(g')$ if $\varphi$ is injective, and $\htop(g)\ge\htop(g')$ if $\varphi$ is surjective.
\item If $Y=Y_1\cup\dots\cup Y_k$ with $Y_i$ closed and $g$-invariant, then $\htop(g)=\max_i\htop(g|_{Y_i})$.
\item $\htop(g^n)=n\htop(g)$ for $n\ge1$.
\item If $\Omega(g)$ is the non-wandering set, then $\htop(g)=\htop(g|_{\Omega(g)})$.
\item (Variational principle.) $\htop(g)=\sup\{h_\mu(g):\mu\in E(Y,g)\}$, where $E(Y,g)$ is the set of ergodic $g$-invariant Borel probability measures and $h_\mu$ is the measure-theoretic entropy.
\end{enumerate}
\end{theorem}

A point $y\in Y$ is an \emph{atom} of a Borel probability measure $\mu$ on $Y$ if $\mu(\{y\})>0$. The \emph{support} $\supp\mu$ of $\mu$ is the set of points all of whose neighbourhoods have positive measure; it is closed, $\mu(\supp\mu)=1$, and a relatively open set of measure zero is disjoint from it. If $\mu$ is $g$-invariant, then $g(\supp\mu)\subset\supp\mu$, because $g^{-1}(\supp\mu)$ is a closed set of full measure and therefore contains $\supp\mu$. The \emph{non-wandering set} $\Omega(g)$ is the set of points $y$ such that every neighbourhood $U$ of $y$ satisfies $g^n(U)\cap U\ne\emptyset$ for some $n\ge1$; it is closed and $g$-invariant.

For a topological space $X$ and $f\in C(X)$ let $\KK(f)$ be the family of non-empty compact sets $K\subset X$ with $f(K)\subset K$, and \cite{CR05}
\[
\ent(f)=\sup\{\htop(f|_K):K\in\KK(f)\}\qquad(\sup\emptyset=0).
\]
If $X$ is compact then $\ent(f)=\htop(f)$. If $\varphi\colon X\to X'$ is a homeomorphism, then $\ent(\varphi f\varphi^{-1})=\ent(f)$, and $\ent(f^n)=n\ent(f)$ \cite{CR05}. If $r\colon X\to A$ is a retraction and $\varphi\in C(A)$, then $\ent(\varphi\circ r)=\ent(\varphi)$: every compact $\varphi$-invariant set is $\varphi\circ r$-invariant, and conversely for $K\in\KK(\varphi\circ r)$ the set $K_\infty=\bigcap_n(\varphi r)^n(K)\subset A$ is compact, invariant, carries the full entropy of $(\varphi r)|_K$ by Theorem~\ref{thm:compactprops}(d) (the non-wandering set of $(\varphi r)|_K$ lies in $K_\infty$), and $(\varphi r)|_{K_\infty}=\varphi|_{K_\infty}$.

A map $f\in C(X)$ is \emph{transitive} if for all non-empty open $U,V$ there is $n\ge1$ with $f^n(U)\cap V\neq\emptyset$, \emph{mixing} if such $n$ exist for all large $n$, and \emph{Devaney chaotic} if it is transitive with dense periodic points. On an interval a transitive map has dense periodic points \cite{Sha64b,AKLS99}, so $\CHD(L)$ is the class of transitive maps of $L$. Let $(X,d)$ be a metric space; a set $S\subset X$ is \emph{scrambled} for $f$ if $\liminf_n d(f^nx,f^ny)=0<\limsup_n d(f^nx,f^ny)$ for all distinct $x,y\in S$, and $f$ is \emph{Li--Yorke chaotic} if it has an uncountable scrambled set. When $S$ lies in a compact invariant set this does not depend on the compatible metric.

An \emph{$s$-horseshoe} for $\varphi\colon L\to\R$ is a family of $s\ge2$ non-degenerate compact intervals $J_1,\dots,J_s\subset L$ with pairwise disjoint interiors such that $J_1\cup\dots\cup J_s\subset\varphi(J_i)$ for all $i$. A map of a compact interval with an $s$-horseshoe has entropy at least $\log s$ \cite{BC92,ALM00}. A piecewise linear map of a compact interval with constant slope $\pm\lambda$, $\lambda\ge1$, has entropy $\log\lambda$ \cite{MS80}.

For $k\ge3$ odd let $\lambda_k$ be the largest root of $x^k-2x^{k-2}-1$; for $n=2^sk$ with $k$ odd put $\sigma_n=1$ if $k=1$ and $\sigma_n=\lambda_k^{2^{-s}}$ otherwise, so that $\sigma_3=\frac{1+\sqrt5}2$ and $\sigma_n>1$ unless $n$ is a power of $2$. We shall use the following results.

\begin{theorem}\label{thm:known}
\begin{enumerate}[label=(\alph*)]
\item \cite{Ste77,BGMY80} If a map $g$ of a compact interval has a point of least period $n$, then $\htop(g)\ge\log\sigma_n$, and there is $g_n\in C(\I)$ with a point of least period $n$ and $\htop(g_n)=\log\sigma_n$.
\item \cite{Sha64} A map of a compact interval with a point of least period $3$ has points of all periods.
\item \cite{Blo82,BC87,Rue17} A transitive map of a compact interval has entropy at least $\frac12\log2$, and a mixing one has entropy greater than $\frac12\log2$.
\item \cite{BC92,Rue17} If $f$ is a transitive map of a compact interval, then either $f$ is mixing, or there is a fixed point $c$ with $f([0,c])=[c,1]$ and $f([c,1])=[0,c]$.
\item \cite{BC92,Rue17} For $s\in(\sqrt2,2]$ let $T_s(x)=s\min\{x,1-x\}$. The core $K_s=[T_s^2(\frac12),T_s(\frac12)]$ is invariant and $T_s|_{K_s}$ is mixing.
\item \cite{BGKM02} A map of a compact metric space with positive entropy has a scrambled set homeomorphic to the Cantor set.
\item \cite{Smi86} There is a Li--Yorke chaotic map of $\I$ with topological entropy $0$.
\end{enumerate}
\end{theorem}

%=====================================================================
\section{The truncation theorem}\label{sec:truncation}
%=====================================================================

The entropy $\ent(f)$ of a map of a non-compact interval is defined through compact invariant sets, which may be hard to locate. The following theorem shows that it can be computed from the truncations $f_J$, which are maps of compact intervals, even though truncation creates orbits (those reaching the endpoints of $J$) that $f$ does not have.

\begin{lemma}\label{lem:atom}
Let $Y$ be a compact metric space, $g\in C(Y)$ and $\mu\in E(Y,g)$. If $\mu$ has an atom, then $\mu$ is supported on a periodic orbit and $h_\mu(g)=0$.
\end{lemma}

\begin{proof}
Let $a=\max_y\mu(\{y\})>0$; the maximum exists because only finitely many points have mass $\ge a/2$. The set $A=\{y:\mu(\{y\})=a\}$ is finite and non-empty, and $\mu(\{gy\})=\mu(g^{-1}\{gy\})\ge\mu(\{y\})=a$ for $y\in A$, so $g(A)\subset A$ and $A$ contains a periodic orbit $P$. Since $g(P)=P$, the sets $g^{-n}(P)$ increase with $n$, their union $P'$ satisfies $g^{-1}(P')=P'$, and ergodicity gives $\mu(P')=1$; as $\mu(g^{-n}P)=\mu(P)$ for all $n$, $\mu(P)=1$. Thus $\mu$ is an invariant measure of the finite system $(P,g)$, whose topological entropy is $0$, and $h_\mu(g)=0$ by the variational principle.
\end{proof}

\begin{theorem}\label{thm:truncation}
\begin{enumerate}[label=(\alph*)]
\item Let $X$ be a Hausdorff space, $Y\subset X$ a compact metrizable subset and $r\colon X\to Y$ a retraction such that $B=r(X\setminus Y)$ is finite. For $f\in C(X)$ put $g=r\circ f|_Y\in C(Y)$. Then $\ent(f)\ge\htop(g)$. More precisely, for every $\mu\in E(Y,g)$ with $h_\mu(g)>0$ the set $S=\supp\mu$ belongs to $\KK(f)$, $f|_S=g|_S$, and $\htop(f|_S)\ge h_\mu(g)$.
\item For every interval $L$ and every $f\in C(L)$,
\[
\ent(f)=\sup\{\htop(f_J): J\subset L\ \text{a non-degenerate compact interval}\}.
\]
\item If $L$ is an interval and $\varphi\in C(L)$ has an $s$-horseshoe, then $\ent(\varphi)\ge\log s$.
\end{enumerate}
\end{theorem}

\begin{proof}
We begin with (a). Note first that $Y$, being a compact subset of a Hausdorff space, is closed in $X$, and that $g=r\circ f|_Y$ is a continuous self-map of the compact metrizable space $Y$; in particular the variational principle (the variational principle) applies to $g$. If $\htop(g)=0$ there is nothing to prove. Otherwise let $\mu\in E(Y,g)$ with $h_\mu(g)>0$; by the variational principle such measures exist, and the supremum of their entropies is $\htop(g)$ (also when $\htop(g)=\infty$).

The measure $\mu$ has no atoms, for otherwise Lemma~\ref{lem:atom} would give $h_\mu(g)=0$; since $B$ is finite, $\mu(B)=0$. Consider the set
\[
E_Y=\{y\in Y: f(y)\notin Y\}=Y\cap f^{-1}(X\setminus Y)
\]
of points at which the retraction actually modifies $f$. Since $X\setminus Y$ is open and $f$ is continuous, $E_Y$ is relatively open in $Y$. For $y\in E_Y$ the point $g(y)=r(f(y))$ lies in $r(X\setminus Y)=B$, so $E_Y\subset g^{-1}(B)$ and, by invariance of $\mu$,
\[
\mu(E_Y)\le\mu\bigl(g^{-1}(B)\bigr)=\mu(B)=0.
\]
For every $n\ge0$ the set $g^{-n}(E_Y)$ is relatively open in $Y$, because $g$ is continuous, and $\mu(g^{-n}(E_Y))=\mu(E_Y)=0$, again by invariance. Recall that $S=\supp\mu$ is the set of points all of whose neighbourhoods have positive measure. If a point $y\in S$ belonged to the relatively open set $g^{-n}(E_Y)$, this set would be a neighbourhood of $y$ in $Y$ of measure zero, which is impossible. Hence
\[
S\cap g^{-n}(E_Y)=\emptyset\qquad\text{for all }n\ge0 .
\]
Taking $n=0$ we see that no point of $S$ lies in $E_Y$; that is, $f(y)\in Y$ for every $y\in S$, and therefore $g(y)=r(f(y))=f(y)$, because $r$ is the identity on $Y$. The support of an invariant measure is forward invariant, so $g(S)\subset S$; consequently $f(S)=g(S)\subset S$. Thus $S$ is a non-empty compact $f$-invariant set, and $f|_S=g|_S$. The measure $\mu$ is concentrated on $S$ and is an ergodic invariant measure of the compact system $(S,g|_S)=(S,f|_S)$, so by the variational principle applied to $f|_S$,
\[
\ent(f)\ge\htop(f|_S)=\htop(g|_S)\ge h_\mu(g).
\]
Taking the supremum over all $\mu\in E(Y,g)$ with positive entropy, and using the variational principle for $g$, we obtain $\ent(f)\ge\htop(g)$.

To prove (b), let $J=[a,b]\subset L$ be a non-degenerate compact interval. The clamp $r_J$ is continuous, it is the identity on $J$, and it maps $L\setminus J$ into the two-point set $\{a,b\}$. Hence $r_J$ is a retraction of $L$ onto the compact metrizable set $J$ with finite residual set, and (a) gives $\ent(f)\ge\htop(f_J)$. Conversely, let $K\in\KK(f)$. Choose a non-degenerate compact interval $J\subset L$ containing $K$; we may take $J=\conv K$ if $K$ has more than one point, and a small compact interval in $L$ containing $K$ otherwise. Since $f(K)\subset K\subset J$, we have $r_J(f(x))=f(x)$ for $x\in K$, so the truncation $f_J$ coincides with $f$ on $K$, and $K$ is $f_J$-invariant. By Theorem~\ref{thm:compactprops}(a), applied to the inclusion of $K$ into $J$, which is injective and conjugates $f_J|_K$ to $f_J$, we get $\htop(f|_K)=\htop(f_J|_K)\le\htop(f_J)$. Taking the supremum over $K\in\KK(f)$ gives the reverse inequality.

For (c), let $J_1,\dots,J_s\subset L$ be an $s$-horseshoe for $\varphi$ and $J=\conv(J_1\cup\dots\cup J_s)\subset L$. Since $r_J$ is the identity on $J$,
\[
\varphi_J(J_i)=r_J(\varphi(J_i))\supset r_J(J_1\cup\dots\cup J_s)=J_1\cup\dots\cup J_s\qquad(1\le i\le s),
\]
so $J_1,\dots,J_s$ is an $s$-horseshoe for the compact interval map $\varphi_J$, whose entropy is therefore at least $\log s$ by the horseshoe bound of Section~\ref{sec:prelim}. By (b), $\ent(\varphi)\ge\htop(\varphi_J)\ge\log s$.
\end{proof}

\begin{figure}[t]
\centering
\begin{tikzpicture}
\begin{axis}[width=8.2cm,height=8.2cm,xmin=-2.35,xmax=2.35,ymin=-3.2,ymax=3.2,axis lines=box,axis on top,
  xtick={-2,-1.5,0,1.5,2},xticklabels={$-2$,$a$,$0$,$b$,$2$},ytick={-1.5,0,1.5},yticklabels={$a$,$0$,$b$},
  every axis plot/.append style={line join=round}]
\fill[gray!12] (axis cs:-1.5,-1.5) rectangle (axis cs:1.5,1.5);
\draw[gray] (axis cs:-1.5,-1.5) rectangle (axis cs:1.5,1.5);
\addplot[domain=-2.3:2.3,samples=300,gray,dashed]{x^3-3*x};
\addplot[domain=-1.5:1.5,samples=400,very thick,blue!70!black]{max(-1.5,min(1.5,x^3-3*x))};
\addplot[domain=-2.3:2.3,samples=2,black!50,dotted]{x};
\draw[red!70!black,line width=2.2pt] (axis cs:-1.3844,-1.5)--(axis cs:-0.5579,-1.5) (axis cs:0.5579,-1.5)--(axis cs:1.3844,-1.5);
\node[font=\small,blue!70!black] at (axis cs:-0.95,2.0) {$f_J$};
\node[font=\small,gray] at (axis cs:1.75,2.75) {$f$};
\node[font=\scriptsize,red!70!black] at (axis cs:0,-1.9) {clipping set $E\subset J$};
\end{axis}
\end{tikzpicture}
\caption{The cubic $f(x)=x^3-3x$ on $\R$ and its truncation $f_J$ to $J=[-1.5,1.5]$. On the clipping set $E$ (where $f$ leaves $J$; flat pieces of $f_J$) the truncation creates new orbits that end on $\{a,b\}$. Theorem~\ref{thm:truncation} says these orbits carry no entropy. By Corollary~\ref{thm:truncation}, $\htop(f_J)\le\ent(f)=\log 3$ for every $J$, with equality as soon as $J\supset[-2,2]$.}
\label{fig:truncation}
\end{figure}
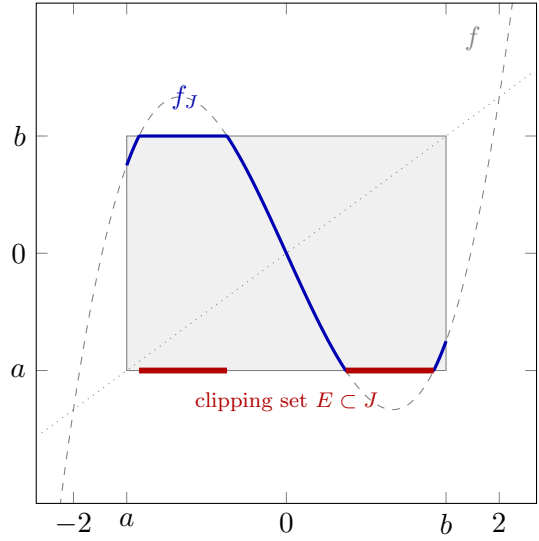

%=====================================================================
\section{Two fixed points force entropy $\log 2$}\label{sec:twofixed}
%=====================================================================

\begin{theorem}\label{thm:interval}
Let $g\in C(\I)$ be transitive.
\begin{enumerate}[label=(\alph*)]
\item If $g$ has at least two fixed points, then $g$ has a $2$-horseshoe and $\htop(g)\ge\log2$.
\item If $g$ fixes an endpoint, then $\htop(g)\ge\log2$.
\item If $\htop(g)<\log2$, then $g$ has exactly one fixed point. If $\htop(g)=\frac12\log2$, then $g$ has exactly one fixed point and $g$ is not mixing.
\end{enumerate}
\end{theorem}

\begin{proof}
We shall use repeatedly three consequences of transitivity. First, $g$ is surjective: $g(\I)$ is a compact interval which meets every non-empty open set, because the orbit of every point of a non-empty open set $U$ enters every non-empty open set $V$ after at least one step, and therefore $g(\I)$ is dense, hence equal to $\I$. Second, no closed interval $K\subsetneq\I$ with non-empty interior satisfies $g(K)\subset K$; otherwise the non-empty open sets $U=\operatorname{int}K$ and $V=\I\setminus K$ would satisfy $g^n(U)\cap V\subset K\cap V=\emptyset$ for all $n$. Third, there is a point $y$ whose orbit is dense in $\I$, and since $\I$ has no isolated points, such a point may be chosen outside any prescribed finite set. We also use the reflection $\rho(x)=1-x$: the map $\tilde g=\rho\circ g\circ\rho$ is transitive, has the same entropy as $g$ (it is conjugate to $g$), has the fixed points $\rho(\Fix(g))$, and satisfies $\tilde g(x)-x=-\bigl(g(1-x)-(1-x)\bigr)$, so it exchanges the roles of the inequalities $g(x)>x$ and $g(x)<x$ and of the endpoints $0$ and $1$.

We prove (a). The set $\Fix(g)$ is closed. It has empty interior: if a non-empty open set $U$ consisted of fixed points, then $g^n(U)=U$ for all $n$, and transitivity would force $U$ to meet every non-empty open set, so $U$ would be dense and, by continuity, $g$ would be the identity, which is not transitive. Let $p<q$ be two fixed points. The open set $(p,q)\setminus\Fix(g)$ is non-empty, and the endpoints of any of its components are fixed points $a<b$ with no fixed point in $(a,b)$. The continuous function $g(x)-x$ does not vanish on $(a,b)$, hence has constant sign there. Replacing $g$ by $\tilde g$ if necessary, we may assume that $g(x)>x$ for all $x\in(a,b)$.

Suppose first that $a>0$. If we had $g(x)\ge a$ for all $x>b$, then, since $g(x)\ge x\ge a$ for $x\in[a,b]$, the interval $[a,1]$, which is a proper closed subinterval with non-empty interior, would be invariant; this is impossible. Hence some $x>b$ satisfies $g(x)<a$. As $g(b)=b>a$, the intermediate value theorem shows that the closed set $\{x\in[b,1]:g(x)=a\}$ is non-empty; let $c$ be its minimum. Then $c>b$. Moreover $g(x)>a$ for every $x\in(a,c)$: for $x\in(a,b]$ this holds because $g(x)\ge x>a$, and for $x\in[b,c)$ because $g(b)>a$ and $g$ does not take the value $a$ on $[b,c)$. Let $z=\max g([a,c])$ and let $d\in[a,c]$ be a point with $g(d)=z$. Since $g\ge a$ on $[a,c]$, if we had $z\le c$, then $g([a,c])\subset[a,c]$, which is impossible because $[a,c]\ne\I$. Hence $z>c$, and $d\in(a,c)$ because $g(a)=g(c)=a<z$. The image of an interval under $g$ is an interval, so
\[
g([a,d])\supset[g(a),g(d)]=[a,z]\supset[a,c],\qquad g([d,c])\supset[g(c),g(d)]=[a,z]\supset[a,c].
\]
The intervals $[a,d]$ and $[d,c]$ are non-degenerate, have disjoint interiors and their union is $[a,c]$; thus they form a $2$-horseshoe (Figure~\ref{fig:twofixed}).

Suppose next that $a=0$ and that $g$ vanishes somewhere on $[b,1]$. Let $c$ be the least zero of $g$ in $[b,1]$; then $c>b$ because $g(b)=b>0$, and $g>0$ on $(0,c)$, by the same argument as before. Let $z=\max g([0,c])=g(d)$. If $c<1$, then $z>c$, since otherwise $[0,c]$ would be a proper invariant interval. If $c=1$, then $z=\max g(\I)=1$ by surjectivity. In both cases $z\ge c$, the point $d$ lies in $(0,c)$, and $g([0,d])\supset[0,z]\supset[0,c]$, $g([d,c])\supset[0,z]\supset[0,c]$, so $[0,d],[d,c]$ is a $2$-horseshoe.

It remains to exclude the possibility that $a=0$ and $g$ has no zero in $[b,1]$. In that case $g>0$ on $(0,1]$, because $g(x)>x>0$ on $(0,b)$. Let $y$ be a point with dense orbit and $y\neq0$. Put $\eta=\min g([b,1])$, which is positive since $g$ is positive on the compact set $[b,1]$, and $m=\min\{y,\eta,b\}>0$. We claim that $g^n(y)\ge m$ for all $n\ge0$. This holds for $n=0$. If $g^n(y)\ge m$, then either $g^n(y)\in[m,b)$, in which case $g^{n+1}(y)>g^n(y)\ge m$ because $g(x)>x$ on $(0,b)$, or $g^n(y)\in[b,1]$, in which case $g^{n+1}(y)\ge\eta\ge m$. Thus the orbit of $y$ misses the non-empty open set $[0,m)$, contradicting its density. Hence $g$ always has a $2$-horseshoe, and $\htop(g)\ge\log2$ by the horseshoe bound of Section~\ref{sec:prelim}.

For (b), by passing to $\tilde g$ we may assume $g(0)=0$. If $g$ had no fixed point in $(0,1]$, then $g(1)<1$, and the continuous function $g(x)-x$, which does not vanish on $(0,1]$ and is negative at $1$, would be negative on all of $(0,1]$. Then $g([0,\frac12])\subset[0,\frac12]$, because $0\le g(x)<x$ for $x\in(0,\frac12]$ and $g(0)=0$; this contradicts transitivity. Hence $g$ has a second fixed point and (a) applies.

Assertion (c) follows from (a) and Theorem~\ref{thm:known}(c): a transitive map with at least two fixed points has entropy at least $\log2$; a transitive map has at least one fixed point, because $g(0)\ge0$ and $g(1)\le1$; and a mixing map has entropy greater than $\frac12\log2$.
\end{proof}

\begin{figure}[t]
\centering
\begin{tikzpicture}
\begin{axis}[width=7.4cm,height=7.4cm,xmin=0,xmax=1,ymin=0,ymax=1,axis lines=box,
  xtick={0.2,0.4,0.6,0.8},xticklabels={$a$,$b$,$d$,$c$},ytick={0.2,0.8,0.95},yticklabels={$a$,$c$,$z$}]
\fill[orange!20] (axis cs:0,0.2) rectangle (axis cs:1,0.8);
\addplot[very thick,blue!70!black] coordinates {(0,0.35) (0.2,0.2) (0.3,0.38) (0.4,0.4) (0.6,0.95) (0.8,0.2) (1,0.6)};
\addplot[domain=0:1,samples=2,black!45,dotted]{x};
\draw[dashed,gray] (axis cs:0.6,0)--(axis cs:0.6,0.95) (axis cs:0.8,0)--(axis cs:0.8,0.2) (axis cs:0,0.95)--(axis cs:0.6,0.95);
\draw[decorate,decoration={brace,amplitude=4pt},thick,red!70!black] (axis cs:0.2,0.04)--(axis cs:0.595,0.04) node[midway,above=4pt,font=\scriptsize]{$J_1$};
\draw[decorate,decoration={brace,amplitude=4pt},thick,red!70!black] (axis cs:0.605,0.04)--(axis cs:0.8,0.04) node[midway,above=4pt,font=\scriptsize]{$J_2$};
\fill (axis cs:0.2,0.2) circle (1.5pt) (axis cs:0.4,0.4) circle (1.5pt);
\end{axis}
\end{tikzpicture}
\caption{Proof of Theorem~\ref{thm:interval}(a) when $a>0$: consecutive fixed points $a<b$ with $g>\mathrm{id}$ on $(a,b)$, the first point $c>b$ with $g(c)=a$, and the maximum $z>c$ of $g$ on $[a,c]$ at $d$. Both $J_1=[a,d]$ and $J_2=[d,c]$ cover the shaded band $[a,c]$.}
\label{fig:twofixed}
\end{figure}
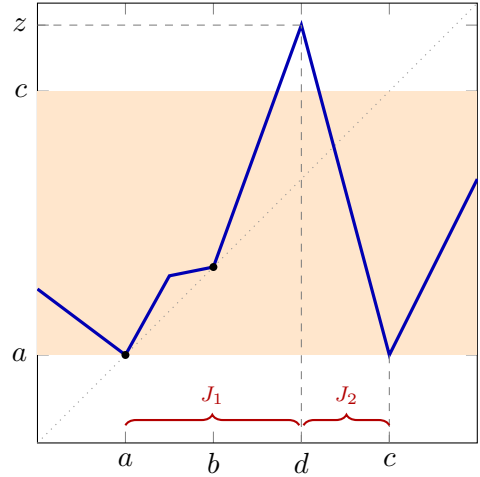

%=====================================================================
\section{Constructions}\label{sec:constructions}
%=====================================================================

\begin{proposition}\label{prop:zigzag}
For $\lambda\ge2$ the map $Z_\lambda(x)=\operatorname{dist}(\lambda x,2\Z)$ is an exact map of $\I$: for every non-degenerate interval $J\subset\I$ there is $m$ with $Z_\lambda^n(J)=\I$ for all $n\ge m$. It is mixing, it fixes $0$ and has at least two fixed points, and $\htop(Z_\lambda)=\log\lambda$.
\end{proposition}

\begin{proof}
Write $Z=Z_\lambda$. The map $Z$ is piecewise linear with slope $\pm\lambda$, and its turning points in $(0,1)$ are the points $k/\lambda$ with $1\le k<\lambda$, where $Z$ takes alternately the values $1$ and $0$. Since $\lambda\ge2$, the interval $[0,\frac1\lambda]$ is mapped by $x\mapsto\lambda x$ onto $\I$, so $Z(\I)=\I$, and every lap $[\frac k\lambda,\frac{k+1}\lambda]$ with $k+1\le\lambda$ is mapped onto $\I$; we call such laps \emph{full}. Let $J$ be a non-degenerate interval. Assume first $\lambda>2$, and put $J_0=J$. Given $J_m$, three situations are possible. If the interior of $J_m$ contains no turning point, $Z$ is affine on $J_m$ and we put $J_{m+1}=Z(J_m)$, so that $|J_{m+1}|=\lambda|J_m|$. If the interior of $J_m$ contains exactly one turning point, it splits $J_m$ into two intervals on each of which $Z$ is affine; the longer one has length at least $\frac12|J_m|$, and we let $J_{m+1}$ be its image, so that $J_{m+1}\subset Z(J_m)$ and $|J_{m+1}|\ge\frac\lambda2|J_m|$. If the interior of $J_m$ contains at least two turning points, then $J_m$ contains a lap whose endpoints are both turning points, which is full, so $Z(J_m)=\I$. In the first two situations $|J_{m+1}|\ge\frac\lambda2|J_m|$ with $\frac\lambda2>1$, and all lengths are at most $1$; hence the third situation occurs for some $m$, and then $Z^{m+1}(J)\supset Z(J_m)=\I$. As $Z(\I)=\I$, we get $Z^n(J)=\I$ for all $n\ge m+1$. If $\lambda=2$, then $Z$ is the tent map, $J$ contains a dyadic interval $[j2^{-k},(j+1)2^{-k}]$, and $Z^k$ maps it affinely onto $\I$; again $Z^n(J)=\I$ for all $n\ge k$. Exactness implies mixing: if $U,V$ are non-empty open sets, $U$ contains a non-degenerate interval $J$ and $Z^n(U)\supset Z^n(J)=\I$ meets $V$ for all large $n$. The point $0$ is fixed, and since $Z$ decreases from $1$ to $0$ on $[\frac1\lambda,\frac2\lambda]\subset\I$, the function $Z(x)-x$ changes sign there, so there is a second fixed point. Finally $\htop(Z)=\log\lambda$ by \cite{MS80}.
\end{proof}

\begin{proposition}\label{prop:Dmap}
There is a mixing map $D_I$ of $\I$ with $D_I(0)=0$, $D_I(1)=1$, infinitely many fixed points and $\htop(D_I)=\infty$.
\end{proposition}

\begin{proof}
For $k\ge0$ put $B_k=[k,k+1]$ and $N_k=k+2$. We define $D_+\colon[0,\infty)\to[0,\infty)$ blockwise. For $k\ge1$ divide $B_k$ into $6N_k+1$ intervals of equal length, called the \emph{pieces} of $B_k$, and prescribe the values of $D_+$ at the $6N_k+2$ division points, from left to right, by the sequence
\[
k,\ \underbrace{k-1,\,k,\,k+1,\,k+2,\,k+1,\,k}_{\text{repeated }N_k\text{ times}},\ k+1 ,
\]
interpolating affinely on each piece. For $k=0$ divide $B_0$ into $4N_0+1$ pieces and use the sequence $0,\ (1,2,1,0)$ repeated $N_0$ times, $1$. Consecutive values in these sequences differ by $1$, so every piece is mapped affinely onto one of the blocks; the value at the left end of $B_k$ is $k$ and at the right end $k+1$, so the definitions agree at common endpoints and $D_+$ is continuous with $D_+(k)=k$ for all $k$. By construction
\[
D_+(B_0)=B_0\cup B_1,\qquad D_+(B_k)=B_{k-1}\cup B_k\cup B_{k+1}\quad(k\ge1).
\]
Each piece has length at most $\frac15$ and is mapped onto an interval of length $1$, so $|D_+'|\ge5$ wherever the derivative exists. For $k\ge1$, each of the $N_k$ repetitions contains the two steps $k\to k+1$ and $k+1\to k$, so $B_k$ contains $2N_k$ pieces, with pairwise disjoint interiors, each of which is mapped onto $B_k$; they form a $2N_k$-horseshoe, and Theorem~\ref{thm:truncation}(d) gives $\ent(D_+)\ge\log(2N_k)$ for every $k$, so $\ent(D_+)=\infty$.

We show that $D_+$ is mixing. Let $J\subset[0,\infty)$ be a non-degenerate compact interval and $J_0=J$. Every interval of length at least $\frac25$ contains a whole piece, since the piece containing its left endpoint and the next piece have total length at most $\frac25$. As in (a), as long as $J_m$ contains no whole piece, its interior contains at most one division point, and we can choose an interval $J_{m+1}\subset D_+(J_m)$ with $|J_{m+1}|\ge\frac52|J_m|$. Hence after finitely many steps some $J_m\subset D_+^m(J)$ contains a piece, whose image is a block $B_j$. From the displayed formula, by induction on $i$, $D_+^i(B_j)\supset B_l$ whenever $|l-j|\le i$, and since $D_+(B_l)\supset B_l$, the same holds for all larger $i$. Now let $U,V$ be non-empty open sets. Choose $J\subset U$ as above and a block $B_l$ whose interior meets $V$. Then $D_+^{m+1+i}(U)\supset D_+^i(B_j)\supset B_l$ meets $V$ for all $i\ge|l-j|$. Thus $D_+$ is mixing.

Finally let $\psi(x)=1-2^{-x}$, a homeomorphism of $[0,\infty)$ onto $[0,1)$, and put $D_I=\psi D_+\psi^{-1}$ on $[0,1)$ and $D_I(1)=1$. Since $D_+(B_k)\subset[k-1,k+2]$, $D_+(x)\to\infty$ as $x\to\infty$, so $D_I$ is continuous at $1$. The set $[0,1)$ is open, dense and invariant, and $D_I$ restricted to it is conjugate to the mixing map $D_+$; as every non-empty open subset of $\I$ meets $[0,1)$ in a non-empty open set, $D_I$ is mixing. It fixes $0$, $1$ and all $\psi(k)$, and $\htop(D_I)\ge\ent(D_I|_{[0,1)})=\ent(D_+)=\infty$.
\end{proof}

\begin{figure}[t]
\centering
\begin{tikzpicture}
\begin{axis}[width=7.2cm,height=7.2cm,xmin=0,xmax=3,ymin=0,ymax=4.1,axis lines=box,xtick={0,1,2,3},ytick={0,1,2,3,4},
 title={\small $D_+$ on $B_0\cup B_1\cup B_2$ (few pieces)},tick label style={font=\scriptsize}]
\fill[blue!6] (axis cs:1,0) rectangle (axis cs:2,3);
\addplot[thick,blue!70!black] table {
0.00000 0
0.20000 1
0.40000 2
0.60000 1
0.80000 0
1.00000 1
1.07692 0
1.15385 1
1.23077 2
1.30769 3
1.38462 2
1.46154 1
1.53846 0
1.61538 1
1.69231 2
1.76923 3
1.84615 2
1.92308 1
2.00000 2
2.07692 1
2.15385 2
2.23077 3
2.30769 4
2.38462 3
2.46154 2
2.53846 1
2.61538 2
2.69231 3
2.76923 4
2.84615 3
2.92308 2
3.00000 3
};
\addplot[domain=0:3,samples=2,black!45,dotted]{x};
\draw[gray!60,dashed] (axis cs:0,1)--(axis cs:3,1) (axis cs:0,2)--(axis cs:3,2) (axis cs:0,3)--(axis cs:3,3);
\end{axis}
\end{tikzpicture}
\qquad
\begin{tikzpicture}
\begin{axis}[width=6.4cm,height=6.4cm,xmin=0,xmax=1,ymin=0,ymax=1,axis lines=box,xtick={0,0.5,1},ytick={0,0.5,1},
 title={\small $g$ built from $G=Z_3$},tick label style={font=\scriptsize}]
\addplot[thick,blue!70!black] coordinates {(0,1) (0.16667,0.5) (0.33333,1) (0.5,0.5) (1,0)};
\addplot[domain=0:1,samples=2,black!45,dotted]{x};
\draw[gray,dashed] (axis cs:0.5,0)--(axis cs:0.5,1) (axis cs:0,0.5)--(axis cs:1,0.5);
\end{axis}
\end{tikzpicture}
\caption{Left: the mixing map $D_+$ of Proposition~\ref{prop:Dmap}, drawn with $N_0=1$, $N_1=N_2=2$ (in the proof $N_k=k+2$). Every piece is mapped affinely onto a block; the block $B_1$ (shaded column) is mapped onto $B_0\cup B_1\cup B_2$ and contains $2N_1$ pieces covering $B_1$. Right: the square-root construction in the proof of the Main Theorem, which halves the entropy of a transitive map $G$ fixing $1$.}
\label{fig:spectra}
\end{figure}
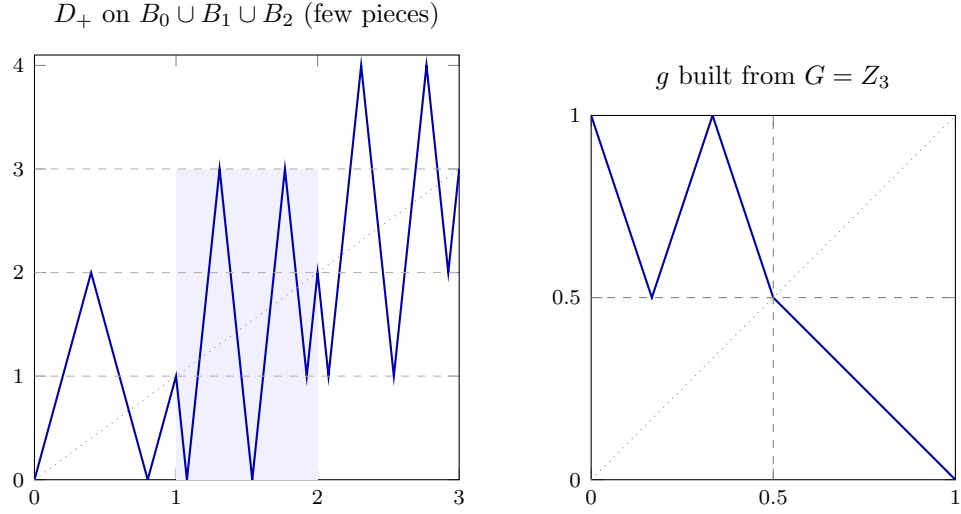

\begin{lemma}\label{lem:attach}
Let $0<c<1$, let $g\in C([0,c])$ and let $A\in C(\I)$ be equal to $g$ on $[0,c]$ and affine on $[c,1]$ with $A(c)=g(c)$ and $A(1)=1$. Then $\htop(A)=\htop(g)$, every periodic orbit of $g$ is a periodic orbit of $A$, and every scrambled set of $g$ is a scrambled set of $A$.
\end{lemma}

\begin{proof}
The interval $[0,c]$ is $A$-invariant and $A=g$ there, which gives the statements on periodic orbits and scrambled sets and the inequality $\htop(A)\ge\htop(g)$. If $g(c)=c$, then $A$ is the identity on $[c,1]$, the sets $[0,c]$ and $[c,1]$ are closed, invariant and cover $\I$, and Theorem~\ref{thm:compactprops}(b) gives $\htop(A)=\max\{\htop(g),0\}=\htop(g)$. Assume now $g(c)<c$. On $[c,1]$ the function $A(x)-x$ is affine, negative at $c$ and zero at $1$, so $A(x)<x$ for $c\le x<1$; moreover $A$ is increasing on $[c,1]$ because $A(1)=1>c>A(c)$. Let $x\in(c,1)$. Choose $\delta>0$ with $c<x-\delta<x+\delta<1$ and $A(x+\delta)<x-\delta$. If $y\in[0,c]$, then $A(y)\in[0,c]\subset[0,x-\delta)$, and if $y\in[c,x+\delta]$, then $A(y)\le A(x+\delta)<x-\delta$. Hence $A([0,x+\delta])\subset[0,x-\delta)$, and by induction $A^m(U)\subset[0,x-\delta)$ for $U=(x-\delta,x+\delta)$ and all $m\ge1$; so $x$ is a wandering point. Consequently the non-wandering set $\Omega(A)$ is contained in $[0,c]\cup\{1\}$. The sets $\Omega(A)\cap[0,c]$ and $\Omega(A)\cap\{1\}$ are closed and invariant, and by Theorem~\ref{thm:compactprops}(b),(h),
\[
\htop(A)=\htop(A|_{\Omega(A)})\le\max\{\htop(g),0\}=\htop(g).\qedhere
\]
\end{proof}

\begin{lemma}\label{lem:glue}
Let $A\in C([0,\frac12])$ and $B\in C([\frac12,1])$ satisfy $A(\frac12)=B(\frac12)=\frac12$, and let $\varphi\in C(\I)$ be equal to $A$ on $[0,\frac12]$ and to $B$ on $[\frac12,1]$. Then $\htop(\varphi)=\max\{\htop(A),\htop(B)\}$. If $L$ is an interval, $J\subset L$ a compact interval and $\eta\colon J\to\I$ an affine homeomorphism, then $f=\eta^{-1}\varphi\eta\circ r_J\in C(L)$ satisfies $\ent(f)=\htop(\varphi)$, and every periodic orbit and every scrambled set of $\varphi$ is mapped by $\eta^{-1}$ onto a periodic orbit and a scrambled set of $f$.
\end{lemma}

\begin{proof}
The sets $[0,\frac12]$ and $[\frac12,1]$ are closed, invariant and cover $\I$, so the first statement is Theorem~\ref{thm:compactprops}(b). The map $r_J$ is a retraction of $L$ onto $J$, so the retraction formula and conjugacy invariance of Section~\ref{sec:prelim} give $\ent(f)=\htop(\eta^{-1}\varphi\eta)=\htop(\varphi)$. Since $f=\eta^{-1}\varphi\eta$ on the invariant set $J$, orbits of points of $J$ are the same for both maps. For scrambled sets note that on the compact set $J$ the metric of $L$ and the metric transported by $\eta$ are uniformly equivalent, so the conditions $\liminf_n d(f^nx,f^ny)=0$ and $\limsup_n d(f^nx,f^ny)>0$ do not depend on the choice.
\end{proof}

\begin{lemma}\label{lem:allvalues}
For every $a\in[0,\infty]$ there is $G\in C(\I)$ with $G(0)=0$ and $\htop(G)=a$.
\end{lemma}

\begin{proof}
For $a<\infty$ take $G=Z_{e^a}$ if $a\ge\log 2$ (Proposition~\ref{prop:zigzag}) and $G(x)=\min\{e^ax,1\}$ otherwise; the latter is piecewise linear with slopes $e^a$ and $0$, so its entropy is at most $a$ by the Lipschitz bound and, if $a>0$, it has the fixed point $1$ and an interval $[0,e^{-a}]$ on which it is expanding by $e^a$, whence $\htop(G)=\lim\frac1n\log\operatorname{Var}(G^n)=a$ by \cite{MS80}; if $a=0$ take $G=\mathrm{id}$. For $a=\infty$ take $G=D_I$.
\end{proof}

%=====================================================================
\section{Proof of the Main Theorem}\label{sec:proof}
%=====================================================================

\begin{proof}[Proof of (ii)]
Let $\cT_{\ge2}$ be the class of transitive maps of $\I$ with at least two fixed points. By Theorem~\ref{thm:interval}(a), $\Spec_{\cT_{\ge2}}(\I)\subset[\log2,\infty]$. By Propositions~\ref{prop:zigzag} and~\ref{prop:Dmap}, the maps $Z_\lambda$, $\lambda\ge2$, and $D_I$ are mixing, have at least two fixed points, and realise every value in $[\log2,\infty]$; hence $\Spec_{\cT_{\ge2}}(\I)=[\log2,\infty]$, and the same holds for mixing maps with two fixed points. The remaining assertions of (ii) are Theorem~\ref{thm:interval}(c).
\end{proof}

\begin{proof}[Proof of (i)]
Since $\CHD(\I)$ is the class of transitive maps of $\I$, Theorem~\ref{thm:known}(c) gives $\Spec_{\CHD}(\I)\subset[\frac12\log2,\infty]$ and $\Spec_\cM(\I)\subset(\frac12\log2,\infty]$, and by (ii) both contain $[\log2,\infty]$. Let $a\in[\log2,\infty]$, and let $G$ be a transitive map of $\I$ with $G(1)=1$ and $\htop(G)=a$; such a map is obtained by conjugating $Z_\lambda$ ($\lambda=e^a$) or $D_I$ with $x\mapsto1-x$, or by taking $D_I$ itself. Define
\[
g(x)=1-\tfrac12G(2x)\ \ (0\le x\le\tfrac12),\qquad g(x)=1-x\ \ (\tfrac12\le x\le1).
\]
Both formulas give $\frac12$ at $x=\frac12$, so $g\in C(\I)$; moreover $g$ maps $[0,\frac12]$ onto $[\frac12,1]$ and $[\frac12,1]$ homeomorphically onto $[0,\frac12]$. For $x\in[0,\frac12]$ we have $g^2(x)=\frac12G(2x)$, so $g^2|_{[0,1/2]}$ is conjugate to $G$ by $x\mapsto2x$; for $y\in[\frac12,1]$ we have $g^2(y)=1-\frac12G(2-2y)$, so $g^2|_{[1/2,1]}$ is conjugate to $G$ as well. We show that $g$ is transitive. Let $U,V$ be non-empty open subsets of $\I$. If $U$ meets $[0,\frac12)$, let $U'=U\cap[0,\frac12)$ and $s=0$; otherwise $U$ meets $(\frac12,1]$, and we let $U'=g(U\cap(\frac12,1])$, which is a non-empty relatively open subset of $[0,\frac12)$, and $s=1$. Similarly, if $V$ meets $[0,\frac12)$, let $V'=V\cap[0,\frac12)$ and $t=0$; otherwise let $V'=(g|_{[0,1/2]})^{-1}(V)$, a non-empty relatively open subset of $[0,\frac12]$, and $t=1$. By transitivity of $g^2|_{[0,1/2]}$ there is $m\ge1$ with $g^{2m}(U')\cap V'\ne\emptyset$, and then $g^{s+2m+t}(U)\cap V\ne\emptyset$. The closed sets $[0,\frac12]$ and $[\frac12,1]$ are $g^2$-invariant and cover $\I$, so Theorem~\ref{thm:compactprops}(b),(c) give $2\htop(g)=\htop(g^2)=\htop(G)=a$. Thus $\Spec_{\CHD}(\I)\supset\frac12[\log2,\infty]\cup[\log2,\infty]=[\frac12\log2,\infty]$, which proves $\Spec_{\CHD}(\I)=[\frac12\log2,\infty]$.

For the mixing spectrum it remains to realise the values in $(\frac12\log2,\log2)$. For $s\in(\sqrt2,2)$ let $T_s(x)=s\min\{x,1-x\}$ and let $K_s=[T_s^2(\frac12),T_s(\frac12)]$ be its core. The interval $K_s$ is invariant, $T_s|_{K_s}$ is mixing (Theorem~\ref{thm:known}(e)), and it is piecewise linear with slope $\pm s$, so $\htop(T_s|_{K_s})=\log s$ by \cite{MS80}. Conjugating by an affine homeomorphism $K_s\to\I$ we obtain mixing maps of $\I$ with entropy $\log s$ for every $s\in(\sqrt2,2)$. Hence $\Spec_\cM(\I)=(\frac12\log2,\infty]$.
\end{proof}

\begin{proof}[Proof of (iv)]
Let $L$ be an interval, fix a compact interval $J\subset L$ and an affine homeomorphism $\eta\colon J\to\I$. Let $n\ge1$ and let $f\in C(L)$ have a point of least period $n$; let $P$ be the orbit of such a point and $J_P=\conv P$. Since $f(P)=P\subset J_P$, the truncation $f_{J_P}$ agrees with $f$ on $P$, so $P$ is an orbit of least period $n$ of the compact interval map $f_{J_P}$; Theorem~\ref{thm:known}(a) gives $\htop(f_{J_P})\ge\log\sigma_n$, and Theorem~\ref{thm:truncation}(b) gives $\ent(f)\ge\log\sigma_n$. Conversely, let $a\in[\log\sigma_n,\infty]$. Let $g_n\in C(\I)$ be a map with a point of least period $n$ and $\htop(g_n)=\log\sigma_n$ (Theorem~\ref{thm:known}(a)), transported affinely to $[0,\frac12]$. Applying Lemma~\ref{lem:attach} with $c=\frac12$ we obtain a map of $\I$ fixing $1$, and conjugating it by $x\mapsto x/2$ we obtain $A\in C([0,\frac12])$ with $A(\frac12)=\frac12$; then $A$ has a point of least period $n$ and $\htop(A)=\log\sigma_n$. By Lemma~\ref{lem:allvalues} there is $G\in C(\I)$ with $G(0)=0$ and $\htop(G)=a$; put $B(x)=\frac12+\frac12G(2x-1)$ for $x\in[\frac12,1]$, so that $B(\frac12)=\frac12$ and $\htop(B)=a$. Lemma~\ref{lem:glue} gives $f\in C(L)$ with $\ent(f)=\max\{\log\sigma_n,a\}=a$ and a point of least period $n$. Hence the spectrum of maps of $L$ with a point of least period $n$ is $[\log\sigma_n,\infty]$.

If $f\in C(L)$ has points of all periods, it has a point of least period $3$, so $\ent(f)\ge\log\sigma_3=\log\frac{1+\sqrt5}2$. Conversely, the map $\varphi$ constructed above with $n=3$ is a map of a compact interval with a point of least period $3$, so by Theorem~\ref{thm:known}(b) it has points of all periods; by Lemma~\ref{lem:glue} so does the corresponding $f\in C(L)$, and $\ent(f)=a$ for any prescribed $a\in[\log\sigma_3,\infty]$.

Finally, let $s\in C(\I)$ be a Li--Yorke chaotic map with $\htop(s)=0$ (Theorem~\ref{thm:known}(g)), transported to $[0,\frac12]$, and let $A\in C([0,\frac12])$ be obtained from it by Lemma~\ref{lem:attach} and the conjugation $x\mapsto x/2$ as above; then $A$ is Li--Yorke chaotic, $A(\frac12)=\frac12$ and $\htop(A)=0$. Gluing $A$ with the map $B$ above, Lemma~\ref{lem:glue} gives, for every $a\in[0,\infty]$, a Li--Yorke chaotic $f\in C(L)$ with $\ent(f)=\max\{0,a\}=a$. Hence $\Spec_{\CHL}(L)=[0,\infty]$.
\end{proof}

\begin{proof}[Proof of (iii)]
If $f\in C(\I)$ has $\htop(f)>0$, then $f$ has an uncountable scrambled set by Theorem~\ref{thm:known}(f), so $\TE(\I)\subset\CHL(\I)$ and $\Spec_{\TE}(\I)=\Spec(\I)\cap(0,\infty]\subset\Spec_{\CHL}(\I)$. By Lemma~\ref{lem:allvalues}, $\Spec(\I)=[0,\infty]$, so $\Spec_{\TE}(\I)=(0,\infty]$; and $\Spec_{\CHL}(\I)=[0,\infty]$ is the case $L=\I$ of (iv). The inclusions are strict because $\frac14\log2\in\Spec_{\TE}(\I)\setminus\Spec_{\CHD}(\I)$ by (i) and $0\in\Spec_{\CHL}(\I)\setminus\Spec_{\TE}(\I)$.
\end{proof}

\begin{remark}
The proof of (i) shows more: every value in $[\frac12\log2,\log2)$ is realised by a transitive map with a unique fixed point which is not mixing, every value in $(\frac12\log2,\infty]$ by a mixing map, and every value in $[\log2,\infty]$ by a mixing map with at least two fixed points. By (ii) these ranges cannot be enlarged: no transitive map of entropy below $\log 2$ has two fixed points.
\end{remark}

\subsection*{Acknowledgements}
The author thanks the management of Muthoot Institute of Technology and Science, Kochi, Kerala, India for providing the facilities to conduct the research.


\begin{thebibliography}{99}\small
\bibitem{ABG24} Ll. Alsedà, L. Bordignon, J. Groisman, \emph{Topological entropy, sets of periods, and transitivity for circle maps}, Ukrain. Mat. Zh. \textbf{76} (2024), 31--47.

\bibitem{AKM65} R. L. Adler, A. G. Konheim, M. H. McAndrew, \emph{Topological entropy}, Trans. Amer. Math. Soc. \textbf{114} (1965), 309--319.
\bibitem{AKLS99} Ll. Alsedà, S. Kolyada, J. Llibre, L. Snoha, \emph{Entropy and periodic points for transitive maps}, Trans. Amer. Math. Soc. \textbf{351} (1999), 1551--1573.
\bibitem{ALM00} Ll. Alsedà, J. Llibre, M. Misiurewicz, \emph{Combinatorial Dynamics and Entropy in Dimension One}, 2nd ed., World Scientific, River Edge, NJ, 2000.
\bibitem{BBCDS92} J. Banks, J. Brooks, G. Cairns, G. Davis, P. Stacey, \emph{On Devaney's definition of chaos}, Amer. Math. Monthly \textbf{99} (1992), 332--334.
\bibitem{BGKM02} F. Blanchard, E. Glasner, S. Kolyada, A. Maass, \emph{On Li--Yorke pairs}, J. Reine Angew. Math. \textbf{547} (2002), 51--68.
\bibitem{BCOT24} J. Bobok, J. Činč, P. Oprocha, S. Troubetzkoy, \emph{Interval maps with dense periodicity}, preprint, arXiv:2402.05638, 2024.

\bibitem{BC87} L. Block, E. M. Coven, \emph{Topological conjugacy and transitivity for a class of piecewise monotone maps of the interval}, Trans. Amer. Math. Soc. \textbf{300} (1987), 297--306.
\bibitem{BC92} L. S. Block, W. A. Coppel, \emph{Dynamics in One Dimension}, Lecture Notes in Math. \textbf{1513}, Springer, Berlin, 1992.
\bibitem{BGMY80} L. Block, J. Guckenheimer, M. Misiurewicz, L.-S. Young, \emph{Periodic points and topological entropy of one dimensional maps}, Lecture Notes in Math. \textbf{819}, Springer, 1980, 18--34.
\bibitem{Blo82} A. M. Blokh, \emph{On sensitive mappings of the interval}, Russian Math. Surveys \textbf{37} (1982), 203--204.
\bibitem{Bow71} R. Bowen, \emph{Entropy for group endomorphisms and homogeneous spaces}, Trans. Amer. Math. Soc. \textbf{153} (1971), 401--414.
\bibitem{CR05} J. S. Cánovas, J. M. Rodríguez, \emph{Topological entropy of maps on the real line}, Topology Appl. \textbf{153} (2005), 735--746.
\bibitem{CV24} A. E. Calderón, E. Villar-Sepúlveda, \emph{Unifying topological entropy notions for piecewise continuous maps on the interval}, preprint, arXiv:2404.10867, 2024.

\bibitem{Dev89} R. L. Devaney, \emph{An Introduction to Chaotic Dynamical Systems}, 2nd ed., Addison-Wesley, 1989.
\bibitem{KW25} D. Kwietniak, F. Wierzbowski, \emph{On mixing and dense periodicity on spaces with a free arc}, preprint, arXiv:2510.15073, 2025.

\bibitem{LY75} T. Y. Li, J. A. Yorke, \emph{Period three implies chaos}, Amer. Math. Monthly \textbf{82} (1975), 985--992.
\bibitem{MS80} M. Misiurewicz, W. Szlenk, \emph{Entropy of piecewise monotone mappings}, Studia Math. \textbf{67} (1980), 45--63.
\bibitem{Rue17} S. Ruette, \emph{Chaos on the Interval}, University Lecture Series \textbf{67}, Amer. Math. Soc., Providence, RI, 2017.
\bibitem{Sha64} A. N. Sharkovskii, \emph{Co-existence of cycles of a continuous mapping of the line into itself}, Ukrain. Mat. Zh. \textbf{16} (1964), 61--71.
\bibitem{Sha64b} A. N. Sharkovsky, \emph{Nonwandering points and the centre of a continuous mapping of the line into itself}, Dopovidi Akad. Nauk Ukrain. RSR (1964), 865--868.
\bibitem{Smi86} J. Smítal, \emph{Chaotic functions with zero topological entropy}, Trans. Amer. Math. Soc. \textbf{297} (1986), 269--282.
\bibitem{Ste77} P. Štefan, \emph{A theorem of Šarkovskii on the existence of periodic orbits of continuous endomorphisms of the real line}, Comm. Math. Phys. \textbf{54} (1977), 237--248.
\bibitem{Wal82} P. Walters, \emph{An Introduction to Ergodic Theory}, Graduate Texts in Math. \textbf{79}, Springer, New York, 1982.
\end{thebibliography}
\end{document}